\documentclass{amsproc}

\usepackage{color}
\usepackage{palatino}
\usepackage{amsrefs}
\usepackage{comment}
\usepackage{tikz-cd}
\usepackage{extarrows}
\usepackage{amsmath,amsthm,amsfonts,amssymb,latexsym,amscd,enumerate}

\newtheorem{theorem}{Theorem}[section]
\newtheorem{lemma}[theorem]{Lemma}
\newtheorem{proposition}[theorem]{Proposition}

\theoremstyle{definition}
\newtheorem{definition}[theorem]{Definition}

\theoremstyle{remark}
\newtheorem{remark}[theorem]{Remark}

\numberwithin{equation}{section}

\usepackage[all]{xy}

\swapnumbers

\DeclareMathOperator{\id}{id}

 \DeclareMathOperator{\tr}{tr}

\newcommand{\forget}[1]{}

\def  \nuint {\raise10pt\hbox{$\nu$}\kern-6pt\int}

\def \Sp {{\cal S}}

\newcommand\Di{D\kern-6pt/}
\newcommand\cDi{{\mathcal D}\kern-6pt/}
\newcommand\spi{S\kern-6pt/}
\newcommand \cspi{\Sp\kern-6pt/}

\def \cal {\mathcal}

\newcommand\Z{\mathbb Z}

 \usepackage{xcolor}
\definecolor{darkgreen}{cmyk}{1,0,1,.2}
\definecolor{m}{rgb}{1,0.1,1}
\DeclareMathOperator{\Ad}{Ad}

    \newcommand{\ka}{\mathfrak{a}}
\newcommand{\kg}{\mathfrak{g}}
\newcommand{\kk}{\mathfrak{k}}

\newcommand{\kp}{\mathfrak{p}}
\newcommand{\kt}{\mathfrak{t}}

\newcommand{\bK}{\mathbb{K}}

\begin{document}

\title{Cartan Motion Group and Orbital Integrals}

%    Information for first author
\author{Yanli Song}
%    Address of record for the research reported here
\address{Department of Mathematics and Statistics, Washington University, St. Louis, Missouri 63130}
%    Current address
%\curraddr{Department of Mathematics and Statistics,
%Case Western Reserve University, Cleveland, Ohio 43403}
\email{yanlisong@wustl.edu}
%    \thanks will become a 1st page footnote.
\thanks{The first author was supported in part by NSF Grants DMS-1800667 and DMS-1952557.}

%    Information for second author
\author{Xiang Tang}
\address{Department of Mathematics and Statistics, Washington University, St. Louis, Missouri, 63130}
\email{xtang@wustl.edu}
\thanks{The second author was supported in part by NSF Grants DMS-1800666 and DMS-1952551.}

%    General info
%\date{January 1, 1994 and, in revised form, June 22, 1994.}

\dedicatory{This paper is dedicated to the occasion of the 40th birthday of cyclic cohomology}


\begin{abstract} In this short note, we study the variation of orbital integrals, as traces on the group algebra $G$, under the deformation groupoid. We show that orbital integrals are continuous under the deformation. And we prove that the pairing between orbital integrals and $K$-theory element of $C^*_r(G)$ stays constant with respect to the deformation for regular group elements, but vary at singular elements.    
\end{abstract}

\maketitle

\section{Introduction}

Connes introduced a beautiful construction of tangent groupoid \cite{connes} to present a groupoid proof of the Atiyah-Singer index theorem. The tangent groupoid provides a deep link between the Fredholm index of an elliptic operator and its local geometric information and has turned out to be extremely powerful in studying various generalized index problems, e.g. \cite{Connes-Skandalis, mont, DLN, vanerp1, vanerp2, VY1, VY2, AMY1, AMY2, PPT}, etc. 

The Connes tangent groupoid was applied  \cite{BCH} to study the $K$-theory of $C^*_r(G)$ and formulate a new approach to the Connes-Kasparov conjecture. More precisely, let $\mathfrak{g}$ and $\mathfrak{k}$ be the Lie algebras of $G$ and $K$, and $\theta$ be the Cartan involution on $\mathfrak{g}$. $\mathfrak{k}$ is the subspace of $\theta$ fixed points and $\mathfrak{p}$ is the $-1$ eigenspace of $\theta$. The adjoint action of $K$ on $\mathfrak{g}$ defines a $K$ action on $\mathfrak{g}/\mathfrak{k}$, which is isomorphic to $\mathfrak{p}$. The Cartan motion group is the semidirect product $K\ltimes \mathfrak{p}$, where $\mathfrak{p}$ is viewed as an abelian group. The deformation groupoid $\mathcal{G}$ is a smooth family of Lie groups over the interval $[0,1]$,
\[
\mathcal{G}:=K\times \mathfrak{p}\times [0,1].
\]
At $t=0$, $G_0$ is the Cartan motion group $K\ltimes \mathfrak{p}$; at $t>0$, $G_t$ is isomorphic to $G$ under the isomorphism $\varphi_t: K\times \mathfrak{p}\to G$ by
\[
\varphi_t(k,v)=k\exp_G(tv),\ (k_1, v_1, t)\cdot_t (k_2, v_2, t)=\varphi^{-1}_t\big(\varphi_t(k_1, v_1)\cdot \varphi_t(k_2, v_2)\big).
\] 
The Connes-Kasparov isomorphism conjecture (\cite{wassermann,nest,lafforgue,CHST}) has the following simple but beautiful formulation, i.e. the deformation groupoid $\mathcal{G}$ defines a natural isomorphism (See Remark \ref{rmk:CK}.)
\[
K_\bullet\big(C^*_r(G_0)\big)\to K_\bullet\big(C^*_r(G_1)\big). 
\]
Higson \cite{higson-mackey} connected Mackey's approach to induced representation theory to the study of the Connes-Kasparov isomorphism conjecture. Afgoustidis \cite{afgoustidis} obtained a new proof of the Connes-Kasparov isomorphism conjecture via the above deformation groupoid using Vogan's theory. 

In this article, we study variations of traces on $C^\infty_c(G)$ under the deformation groupoid. The traces we consider are from orbital integrals.

Let $G$ be a connected real reductive Lie group,  $K$ be a maximal compact subgroup of $G$. For an element $x\in K$, let $G_x$ be the centralizer subgroup of $G$ associated to $x$. The orbital integral is a trace $\tau_x$ on $C^\infty_c(G)$ defined by 
\[
\tau_x(f):=\int_{G/G_x} f(g\theta g^{-1}) dg. 
\]
The orbital integral $\tau_k$ plays a fundamental role in Harish-Chandra's theory of Plancherel measure. 

Let $C^*_r(G)$ be the reduced group $C^*$-algebra of $G$. Recently, Hochs and Wang \cite{HW} showed that $\tau_x$ defines a trace on the Harish-Chandra Schwartz algebra $\mathcal{C}(G)$ and applied the orbital integrals to study the $G$-index of an invariant Dirac operator on a proper cocompact manifold. They proved a delocalized index theorem to detect interesting information of the Connes-Kasparov index map and the $K$-theory of the reduced group $C^*$ algebra $C^*_r(G)$ when $G$ has discrete series representations. For a general $G$, we \cite{ST1} introduced a generalization of the orbital integral $\tau_x$ to a higher degree cyclic cocycle, called higher orbital integral, associated to a cuspidal parabolic subgroup $P$ of $G$ and extract the character information of limits of discrete series representations from the pairing between higher orbital integrals and $K_0(C^*_r(G))$, and with Hochs \cite{HST} obtained a higher index theorem for the pairing between the higher orbital integrals and the $G$-index of an invariant Dirac operator. 

In this paper, we focus on the classical orbital integrals $\tau_x$. We start with observing  the following continuity property of the orbital integral $\tau_x$ on the deformation groupoid $\mathcal{G}$.\\

\noindent{\bf Theorem.} (Theorem \ref{main thm})
{\em Suppose that $f$ is a smooth  function on $\mathcal{G}$ with compact support and $x \in K$ is a regular element.  Then 
\[
\lim_{t \to 0} \int_{G_t} f(g \cdot x \cdot  g^{-1}, t) \; d_t g = \int_K  \int_\kp f\left(kx k^{-1}, w-\text{Ad}_{kx k^{-1}}   w , 0\right ) \; dk dw .
\]}

Next we study the pairing between orbital integrals $\tau_x$ and $K_0$ group and prove the following rigidity property of the index pairing.\\

\noindent{\bf Theorem.} {\em (Theorem \ref{t0 thm}) Let $E$ be an irreducible $K$ representation and  $P^E_t$ an element of $K_0(C^*_r(G_t))$ associated to $E$. The index paring 
$\tau_x \left([P^E_t]\right)  \in \mathbb{C}$ is independent of $t$. When $t = 0$, we have that 
\[
\tau_x \left([P^E_0]\right)  = (-1)^{\frac{\dim \kp}{2}} \cdot \frac{\sum_{w \in W_K} (-1)^w \cdot e^{w(\mu + \rho)}(x)}{\prod_{\alpha \in \Delta^+(\kg, \kt)}\left( e^{\frac{\alpha}{2}} - e^{\frac{-\alpha}{2}} \right)(x)}.
\]	
}

It is worth pointing out that it is crucial to consider a regular element $x$ in the above theorem. In contrast, we show in Theorem \ref{thm:L2}  that the $L^2$-trace does vary with respect to the parameter $t$ and vanishes at $t=0$ with order $\dim(\mathfrak{p})$.

The above theorems suggest that it is possible to use the deformation groupoid to study the limit of the pairing between the higher cyclic cocycles introduced \cite{PPT} and \cite{ST1}  on $K_\bullet(C^*_r(G))$. Such a study is expected to exhibit more interesting properties of the higher cyclic cocycles on the Harish-Chandra Schwartz algebra $\mathcal{C}(G)$. We plan to study this problem in the near future. 

The article is organized as follows. In Section \ref{sec:semiproduct}, we briefly introduce the Cartan motion group, its group $C^*$-algebra and Fourier transform; in Section \ref{sec:limit}, we investigate the variation of orbital integrals with respect to the deformation groupoid and prove Theorem \ref{main thm}; in Section \ref{sec:pairing}, we study the variation of the pairing between orbital integrals and $K_0$ groups and prove Theorem \ref{t0 thm} and \ref{thm:L2}.\\

\noindent{\bf Acknowledgments.} We would like to thank Alexandre Afgoustidis, Wushi Gold-ring, Nigel Higson, Yuri Kordyukov, Shiqi Liu, Ryszard Nest, Sanaz Pooya,  and Sven Raum for inspiring discussions.

\section{Cartan Motion Group}\label{sec:semiproduct}
Let $K$ be a connected compact Lie group and $\kp$ is an even dimensional Eculid-ean space on which $K$ acts. For simplicity, we assume that the $K$-action on $\kp$ is spin so that we have the following diagram.
\[
 \begin{tikzcd}
   & \text{spin}(\kp) \arrow[d] \\
 K \arrow[ru]\arrow[r] & SO(\kp)
   \end{tikzcd}
    \]
Let $S_\kp = S_\kp^+ \oplus S_\kp^-$ be the $\mathbb{Z}_2$-graded spin representation of $\mathrm{Cliff}(\kp)$. In addition, we assume that the $K$-fixed part of $\kp$ is trivial.  

\begin{definition}
The semidirect product $K \ltimes \kp$ is the group whose underlying set is the Cartesian product $K \times \kp$, equipped with the product law	
\[
(k_1, v_1) \cdot_0 (k_2, v_2) = \left(k_1k_2, k_2^{-1}\cdot v_1 +  v_2\right), \quad k_1, k_2 \in K, \quad  \text{and} \quad v_1, v_2 \in \kp,
\]
which is called the \emph{Cartan motion group}. 
\end{definition}

\begin{definition}
The \emph{reduced group $C^*$-algebra} of $K \ltimes \kp$, denoted by $C^*_r(K \ltimes \kp)$, is the completion of the convolution algebra $C_c^\infty(K \ltimes \kp)$ in the norm obtained from the left regular representation of $C_c^\infty(K \ltimes \kp)$ as bounded convolution operators on the Hilbert space $L^2(K \ltimes \kp)$. 
\end{definition}

Fix the Haar measure on $K$ so that the volume of $K$ equals one. Together with a $K$-invariant measure on vector space $\kp$, we obtain a Haar measure on $K \ltimes \kp$. If $f$ is a smooth, compactly supported function on $ K \ltimes \kp$, we define its \emph{Fourier transform}, a smooth function from $\widehat{\kp}$ into the smooth functions on $K \times K$ by the formula
\[
\widehat{f}(\varphi) (k_1, k_2) = \int_\kp f(k_1 k_2^{-1},  v) \varphi(k_2^{-1}\cdot v) \; dv.
\]
Here $\varphi \in \widehat{\kp}, k_1, k_2 \in K, v \in \kp$ and $e$ is the identity element in $K$. 

For a fixed $\varphi \in \widehat{\kp}$, we shall think of $\widehat{f}(\varphi)$ as an integral kernel and hence as a compact operator on $L^2(K)$. The Fourier transform $\widehat{f}$ is therefore a function from $\widehat{\kp}$ into $\mathbb{K}\left(L^2(K)\right)$, where $\mathbb{K}\left(L^2(K)\right)$ denotes the space of compact operators on $L^2(K)$. The function is equivariant for the natural action of $K$ on $\widehat{\kp}$ and the conjugacy action of $K$ on $\mathbb{K}\left(L^2(K)\right)$ induced from the right regular representation of $K$ on $L^2(K)$. We define
\[
\begin{aligned}
&C_0\left(\widehat{\kp}, \mathbb{K}\left(L^2(K)\right)\right)\\
 =& \left\{ \text{continuous functions from $\widehat{\kp}$ into $\mathbb{K}\left(L^2(K)\right)$,  vanishing at infinity} \right\}
\end{aligned}
\]
and 
\[
C_0\left(\widehat{\kp}, \mathbb{K}\left(L^2(K)\right)\right)^K = \left\{\text{$K$-equivariant functions in $C_0\left(\widehat{\kp}, \mathbb{K}\left(L^2(K)\right)\right)$} \right\}.
\]

The Fourier transform provides a concrete description of the  (reduced) group $C^*$-algebra $C^*_r(K\ltimes \mathfrak{p})$.
\begin{theorem}
The Fourier transform $f \mapsto \widehat{f}$ extends to a $C^*$-algebra isomorphism 	
\[
C_r^*(K \ltimes \kp)  \cong C_0\left(\widehat{\kp}, \mathbb{K}\left(L^2(K)\right)\right)^K
\]
\end{theorem}
\begin{proof}
See \cite[Theorem 3.2]{higson-mackey}.	
\end{proof}

We study the \emph{orbital integrals} on Cartan motion group. Let $T\subseteq K$ be the maximal torus of $K$ and $x \in T$ be a regular element. In particular, $x$ gives an endomorphism from $\kp$ to itself such that $\det_\kp(\id - x) \neq 0$. 
We write $\mathcal{O}_x$ the conjugacy class of $(x, 0)\in K \ltimes \kp$.

\begin{definition}
For any $f \in C_c(K \ltimes \kp)$, we define the orbital integral 
\[
\tau_x(f) \colon = \int_{(k, v) \in \mathcal{O}_x} f(k, v)\; dkdv. 
\]	
\end{definition}

\begin{lemma}
\label{integral lem}
If $x$ is regular in $T$, then the orbital integral associated to $x$ for the Cartan motion group equals 
\[
\begin{aligned}
\tau_x(f) = \frac{1}{\mathrm{det}_\kp(\id - x )} \cdot \int_K \int_\kp f(k\cdot x \cdot k^{-1}, v ) \; dk dv.
\end{aligned}
\]
\end{lemma}
\begin{proof}
For any $(k, v) \in K \ltimes \kp$, we have that  
\[
\begin{aligned} 
(k, v) \cdot_0 (x, 0 ) \cdot_0 (k, v)^{-1} = &(k, v) \cdot_0 (x, 0 ) \cdot_0 (k^{-1}, -k\cdot v)\\
=& (k\cdot x, x^{-1}\cdot v ) \cdot_0 (k^{-1}, -k \cdot v)\\
=&(k\cdot x \cdot k^{-1}, kx^{-1}\cdot v  - k \cdot v).
\end{aligned}
\]
The lemma follows immediately from change of variable 
\[
v \mapsto kx^{-1}\cdot v  - k\cdot v=kx^{-1}(\id -x) \cdot v. 
\]
\end{proof}

For any finite dimensional $K$-representation $E$, we denote by $\chi_E$ the character of $E$. 

\begin{lemma}
The determinant function $\det_\kp(\id-x)$ has the following formula, 	
\[
\mathrm{det}_\kp(\id - x ) = (-1)^{\frac{\dim \kp}{2}} \cdot\left(\chi_{S_\kp^+}(x) -\chi_{S_\kp^-}(x) \right)^2,
\]
where $\chi_{S_\kp^\pm}$ is the character of $S_\kp^\pm$. 
\end{lemma}
\begin{proof}
Let $\kt$ be the Lie algebra of $T$ and $\Delta^+(\kp, \kt)$ be a fixed  set of positive roots for the $T$-action on $\kp$. Suppose that $x = e^H$ with  $H\in \kt$. We compute that 
\[
\begin{aligned}
\text{det}_\kp(\text{id} -  e^H)=&\prod_{\alpha\in \Delta^+(\kp, \kt)}\left(1 - e^{\alpha(H)}\right) \cdot \left(1 - e^{-\alpha(H)}\right)\\
=&(-1)^{\frac{\dim \kp}{2}} \cdot \prod_{\alpha\in \Delta^+(\kp, \kt)}\left(e^{\frac{\alpha(H)}{2}} - e^{-\frac{\alpha(H)}{2}}\right)^2.
\end{aligned}
\]	

\end{proof}

\begin{proposition}
\label{prop tau}
Let $E$ be a finite dimensional $K$-representation and  $x$ be a regular element in $T$. Suppose that $g(z)$ is a $K$-invariant Schwartz function on $\widehat{\kp}$ and  $f \in C_0(K \ltimes \kp)$ is the inverse Fourier transform of 
\[
g(z) \cdot \id_E \in  C_0 \left(\widehat{\kp}, \bK \left(L^2(K)\right) \right)^K, \quad z \in \widehat{\kp},
\]
where by the Peter-Weyl theorem $\id_E$ is an element of $\bK \left(L^2(K)\right)$.
Then  
\[
\tau_x(f) = (-1)^{\frac{\dim \kp}{2}} \cdot \frac{\chi_E(x) }{\left(\chi_{S_\kp^+}(x) -\chi_{S_\kp^-}(x) \right)^2} \cdot g(0). 
\]	
\begin{proof}
By Lemma \ref{integral lem},  
\[
\begin{aligned}
\tau_x(f) =& \frac{1}{\text{det}_\kp(\id - x)} \cdot \int_K \int_\kp f(k\cdot x \cdot k^{-1}, v ) \; dk dv. 
\end{aligned}
\]	
For any regular $z\in \widehat{\kp}$, it determines a $K \ltimes \kp$-representation on $L^2(K)$ given by 
\[
\left(\pi_z(x, v)s\right)(u) = e^{i\langle \mathrm{Ad}_uz, v\rangle } \cdot s(x^{-1}u), \quad u \in K,
\]
where $s \in L^2(K)$. By the Fourier transform, we have that 
\[
\begin{aligned}
f(k\cdot x \cdot k^{-1}, v ) =& \int_{\widehat{\kp}}g(z) \cdot \tr\left(\pi_z(kx k^{-1}, v) \circ \id_E \right) \; dz\\
=&\chi_E(x) \cdot \int_{\widehat{\kp}}\int_K g(z) \cdot e^{i\langle \mathrm{Ad}_uz, v\rangle } \; du dz
\end{aligned}
\]
Therefore, 
\[
\begin{aligned}
\tau_x(f) =&  \frac{\chi_E(x)}{\text{det}_\kp(\text{id} - x)} \cdot \int_K \int_\kp \int_{\widehat{\kp}}\int_K g(z) \cdot e^{i\langle \mathrm{Ad}_uz, v\rangle } \; du dz dk dv\\
=& (-1)^{\frac{\dim \kp}{2}} \cdot\frac{\chi_V(x) }{\left(\chi_{S_\kp^+}(x) -\chi_{S_\kp^-}(x) \right)^2} \cdot \int_\kp\int_{\widehat{\kp}}g(z) \cdot e^{i \langle z, v\rangle} \; dzdv\\
=& (-1)^{\frac{\dim \kp}{2}} \cdot\frac{\chi_V(x) }{\left(\chi_{S_\kp^+}(x) -\chi_{S_\kp^-}(x) \right)^2} \cdot g(0). 
\end{aligned}
\]
\end{proof}

\end{proposition}

\section{Deformation to the Cartan motion groups}\label{sec:limit}
In this section, we study the limit of the orbital integral under the deformation to the Cartan motion groups. Let $G$ be a connected, linear,  real reductive Lie group and $K$ be its maximal compact subgroup. Let $\kg$ and $\kk$ be the Lie algebra of $G$ and $K$, and $\kg = \kk \oplus \kp$ be the Cartan decomposition. Let $dg$ be the standard Haar measure on $G$. Let $\ka$ be the maximal abelian sub-algebra in $\kp$ and $\ka^+$ the positive Weyl chamber. 

\begin{lemma}
The following identity holds
\[
\int_G f(g) dg = \int_K \int_{\ka^+} \int_K f\left(k\exp(v)k'\right) \cdot \prod_{\alpha \in \Delta(\kp, \ka)} \left(\sinh(\alpha(v))\right)^{n_\alpha}\; dk dvdk',
\]	
where $\Delta(\kp, \ka)$ is the set of restricted roots of $\ka$-action on $\kp$ and $n_\alpha$ is the multiplicity. 
\end{lemma}
\begin{proof}
See \cite[Proposition 5. 28]{kanpp-book}. 	
\end{proof}

For every non-zero real number $t$, we define a group $G_t$ by using the global diffeomorphism $\varphi_t \colon K \times \kp \to G$ given by 
\[
(k,v) \mapsto k \exp_G(tv).
\]
We define a family of groups 
\[
\mathcal{G} = \begin{cases}
 G_t, & t > 0,\\
 &\\
 G_0 = K \ltimes \kp, & t = 0.
 \end{cases}
\]
The bijection
\[
K \times \kp \times [0, 1] \to \mathcal{G}
\]
defined by the formula
\[
(k, v, t) \mapsto \begin{cases}
(k, v, 0), & t = 0,\\
 &\\
\left(k\exp(tv), t \right), & t > 0, 
 \end{cases}
\]
is a diffeomorphism. As a result, every smooth function $f$ on $\mathcal{G}$ has the following form
\[
\begin{cases}
f(k, v, 0) = F(k, v, 0),\\
 &\\
f(k\exp(tv), t) = F(k, v, t), & t > 0, 
 \end{cases}
\]
for some smooth function $F$ on $K \times \kp \times [0, 1]$. 

\begin{definition}\label{defn:haar-measure-t}
For every positive real number $t$, we define the Haar measure $d_t g$ on $G_t$ by the following formula	
\[
\begin{aligned}
\int_{G_t} f(g, t) \; d_t g &= \frac{1}{t^{\dim \kp - \dim \ka}}\int_K \int_{\ka^+} \int_K f\left(k\exp(tv)k', t\right)\\
&\qquad \qquad \qquad\cdot \prod_{\alpha \in \Delta(\kp, \ka)} \left(\sinh(\alpha(tv))\right)^{n_\alpha}\; dk dvdk'.
\end{aligned}
\]
\end{definition}

\begin{theorem}
\label{main thm}
Suppose that $f$ is a smooth  function on $\mathcal{G}$ with compact support and $x \in K$ is a regular element.  Then 
\[
\lim_{t \to 0} \int_{G_t} f(g \cdot x \cdot  g^{-1}, t) \; d_t g = \int_K  \int_\kp f\left(kx k^{-1}, w-\text{Ad}_{kx k^{-1}}   w , 0\right ) \; dk dw .
\]
\end{theorem}

\begin{proof}
By Definition \ref{defn:haar-measure-t} of Haar measure on $G_t$, we have that 
\[
\begin{aligned}
&\int_{G_t} f(g \cdot x \cdot  g^{-1}, t) \; d_t g \\
=& \int_K \int_{\ka^+} \int_K 
f\left(k'\exp(tv)kx k^{-1} \exp(-tv)k'^{-1}, t\right ) \\
&\qquad \qquad \qquad \cdot \prod_{\alpha \in \Delta(\kp, \ka)} \left(\frac{\sinh(\alpha(tv))}{t}\right)^{n_\alpha}\; dk dvdk'\\
=& \int_K \int_{\ka^+} \int_K 
f\left(\exp(t \Ad_{k'}v)k'kx k^{-1} \exp(-tv)k'^{-1}, t\right ) \\
&\qquad \qquad \qquad \cdot \prod_{\alpha \in \Delta(\kp, \ka)} \left(\frac{\sinh(\alpha(tv))}{t}\right)^{n_\alpha}\; dk dvdk'\\
=& \int_K \int_{\ka^+} \int_K 
f\left(\exp(t \Ad_{k'}v)\exp(-t\Ad_{k'kx k^{-1} k'^{-1}}\circ Ad_{k'}v)k'kx k^{-1} k'^{-1}, t\right ) \\
&\qquad \qquad \qquad \cdot  \prod_{\alpha \in \Delta(\kp, \ka)} \left(\frac{\sinh(\alpha(tv))}{t}\right)^{n_\alpha}\; dk dvdk'\\
=& \int_K \int_{\ka^+} \int_K 
f\left(\exp(t \Ad_{k'}v)\exp(-t\Ad_{kx k^{-1}}\circ Ad_{k'}v)kx k^{-1} , t\right ) \\
&\qquad \qquad \qquad \cdot \prod_{\alpha \in \Delta(\kp, \ka)} \left(\frac{\sinh(\alpha(tv))}{t}\right)^{n_\alpha}\; dk dvdk',
\end{aligned}
\]	
where in the last equation we have changed $k'k$ to $k$. Let $w = \Ad_{k'}v \in \kp$. Then 

\[
\begin{aligned}
&\lim_{t \to 0} f\left(\exp(t \Ad_{k'}v)\exp(-t\Ad_{kx k^{-1}}\circ Ad_{k'}v)kx k^{-1} , t\right ) \\
=&f\left(kx k^{-1}, w-\Ad_{kx k^{-1}}w , 0\right ) .
\end{aligned}
\]
Moreover, we have 
\[
\lim_{t \to 0} \frac{\sinh(\alpha(tv))}{t} = \alpha(v). 
\]
Hence, by the compact support assumption on $f$, we can commute the integral with the limit with respect to $t$ and compute
\[
\begin{aligned}
&\lim_{t \to 0}\int_{G_t} f(g \cdot x \cdot  g^{-1}, t) \; d_t g \\
=&	 \int_K \int_{\ka^+} \int_K f\left(kx k^{-1}, \Ad_{k'}v-\Ad_{kx k^{-1}} \circ \Ad_{k'}v , 0\right ) \cdot 
\prod_{\alpha \in \Delta(\kp, \ka)} \left(\alpha(v)\right)^{n_\alpha} dk dv dk'\\
=& \int_K  \int_\kp f\left(kx k^{-1}, w-\Ad_{kx k^{-1}} w , 0\right ) \; dw dk,
\end{aligned}
\]
where the last equation follows from the fact that the term 
\[
\prod_{\alpha \in \Delta(\kp, \ka)} \left(\alpha(v)\right)^{n_\alpha}
\]
is the Jacobian of the map $\ka^+ \times K \to \kp$.  
\end{proof}

As we assume in Section \ref{sec:semiproduct} that the $K$ action on $\mathfrak{p}$ is spin, $G_t/K$ is equipped with a $G_t$ invariant spin structure. Let $D_{G_t/K}$ be the associated Dirac operator on the homogeneous space $G_t/K \cong G/K$ for $t \geq 0$. Let $\kappa_t^s$ be the smoothing kernel of the heat operator $e^{-sD^2_{G_t/K}}$. It is known \cite{HC75} that $\kappa_t^s$ gives a Harish-Chandra Schwartz function on $G_t$.  

\begin{lemma}
We have the following
\[
\lim_{t \to 0} \kappa_t^s(k\exp(tv)) = \kappa_0^s(k, v). 
\]	
\end{lemma}
\begin{proof} 
It is proved in \cite[Theorem 2.48]{BGV} that the heat kernel on a compact manifold is smooth with respect to the parameter $t$, indexing a smooth family of Riemannian metrics. The same proof extends to $G_t/K$ with analogous estimates.  See \cite[Lemma 2.2, Theorem 2.2]{GR}.
\end{proof}

 We define a function $\Phi\colon \mathcal{G} \to C$ by 
\[ 
\begin{cases}
\Phi(k, v, 0) = s^{-\dim G/2} \cdot \exp\left(\frac{-|v|^2}{16s}\right)
,\\
 &\\
\Phi(k\exp(tv), t) = s^{-\dim G/2} \cdot \exp\left(\frac{-|v|^2}{16s}\right)
, & t > 0. 
 \end{cases}
\]
By the estimate in \cite{Li-Yau}, there exists a positive constant $C$ such that, for all $t\in [0, 1]$, 
\begin{equation}
\|\kappa_t^s(g)\|\leq C \cdot \Phi(g), \quad g \in G_t. 
\end{equation}
By the same computation in Theorem \ref{main thm}, as $\Phi(g)$ is independent of $t$, the integral
\[
\int_{G_t} f(g \cdot x \cdot  g^{-1}, t) \; d_t g 
\]
is uniformly bounded for all $t \in [0, 1]$. Hence, the following theorem 
follows from the dominated convergence theorem:
\begin{theorem}
\label{contin thm}
For any $s > 0$, 
\[
\lim_{t \to 0} \int_{G_t} \kappa^s_t(g \cdot x \cdot  g^{-1}, t) \; d_t g = \int_K  \int_\kp \kappa^s_0\left(kx k^{-1}, w-\text{Ad}_{kx k^{-1}}   w , 0\right ) \; dk dw .
\]	
\end{theorem}

\section{$K$-theory Paring}\label{sec:pairing}
We have defined the deformation groupoid $\mathcal{G}$, a family of Lie groups $\{G_t\}_{0\leq t \leq 1}$, in last section. We can consider the reduced $C^*$-algebra of each $G_t$ for the corresponding Haar measure. The field
\[
\left\{ C^*_r(G_t) \right\}_{0 \leq t \leq 1}
\]
is then a continuous field of $C^*$-algebra \cite[Section 6.2]{higson-mackey}. 

Let $E$ be an irreducible $K$-representation and $S_\kp = S_\kp^+ \oplus S_\kp^-$ be the $\Z_2$-graded spin module in Section \ref{sec:semiproduct}. We consider a family of homogeneous spaces $G_t/K$ and a family of Dirac operators $\{D^E_t\}_{0 \leq t \leq 1}$ acting on 
\[
\left[L^2(G_t) \otimes S_\kp \otimes E\right]^K. 
\]
We denote by $D^E_{t,\pm}$ its restriction to 
\[
\left[L^2(G_t) \otimes S^\pm_\kp \otimes E\right]^K. 
\]
We consider the following Connes-Moscovici projection \cite{connes-moscovici}
\[
P^E_t = \begin{pmatrix}
	e^{-D^{E}_{t,-}D^{E}_{t,+}} & e^{\frac{-D^{E}_{t,-}D^{E}_{t,+}}{2}} \cdot \frac{(1 - e^{-D^{E}_{t,-}D^{E}_{t,+}})}{D^{E}_{t,-}D^{E}_{t,+}} \cdot D^{E}_{t,-}\\
	e^{-\frac{D^{E}_{t,+}D^{E}_{t,-}}{2}} \cdot D^{E}_{t,+} & 1 - e^{-D^{E}_{t,+}D^{E}_{t,-}} 
\end{pmatrix}
\]
which is an idempotent and 
\[
[P^E_t] - \begin{pmatrix}
	0& 0\\
	0& 1
\end{pmatrix}
\]
defines a generator in $K\left(C_r^*(G_t) \right)$. The following is the well-known Connes-Kasparov conjecture proved by Lafforgue \cite{lafforgue}. 
\begin{theorem}[Connes-Kasparov]\label{thm:CK}
If $G$ is spin, then the map 
\[
R(K)\to K^*(C^*_rG) \colon [E] \to [P^E_t] - \begin{pmatrix} 
	0& 0\\
	0& 1
\end{pmatrix} 
\]	
is an isomorphism, where $R(K)$ denotes the character ring of $K$.  
\end{theorem}

\begin{remark}\label{rmk:CK}
One can find different approaches to the Connes-Kasparov conjecture in \cite{afgoustidis, BCH, CHST, higson-mackey, wassermann }. Let us point out the Afgoustidis' proof is based on the study of Cartan motion group suggested in \cite{BCH, higson-mackey}. To be more precise,  we consider 
\[
\mathcal{C} \colon = C^*\text{-algebra of continuous sections of $C^*_r(G_t)$}.
\]
The evaluation maps at $t = 0$ and $t = 1$ induce $C^*$-algebra morphism from $\mathcal{C}$ to $C^*_r(G_0)$ and $C^*_r(G)$, respectively, and in turn induce two isomorphisms
\[
\alpha_0 \colon K(\mathcal{C}) \to K(C^*_r(G_0)), \quad \alpha_1 \colon K(\mathcal{C}) \to K(C^*_r(G)). 
\] 	
On the other hand, it is easy to show that $ K(C^*_r(G_0)) \cong R(K)$. Chabert-Echterhoff-Nest \cite{nest} generalized the above Theorem \ref{thm:CK} to almost connected Lie groups. 
\end{remark}

For any regular element $x \in T$, the orbital integral $\tau_x$ is a trace on $\mathcal{C}(G_t)$, 
\[
\tau_x(f \star g) = \tau_x (g \star f),
\]
where $\star$ denotes the convolution product on $\mathcal{C}(G)$. Thus, it induces a continuous (by Theorem \ref{main thm}) family of traces
\[
\tau_x \colon K(C^*_r(G_t)) \to \mathbb{C}. 
\]

\begin{theorem}
Let $E$ be an irreducible $K$-representation with highest weight $\mu$. For all $t >0$, we have that 
\[
\tau_x \left([P^E_t]\right)  = (-1)^{\frac{\dim \kp}{2}} \cdot \frac{\sum_{w \in W_K} (-1)^w \cdot e^{w(\mu + \rho_K)}(x)}{\prod_{\alpha \in \Delta^+(\kg, \kt)}\left( e^{\frac{\alpha}{2}} - e^{\frac{-\alpha}{2}} \right)(x)},
\]	
which is the character for discrete series representation of $G$ with Harish-Chandra parameter $\mu$. 
\end{theorem}

\begin{proof}
See \cite[Theorem 3.1]{HW}. 	
\end{proof}

\begin{theorem}
\label{t0 thm}
When $t = 0$, we have that 
\[
\tau_x \left([P^E_0]\right)  = (-1)^{\frac{\dim \kp}{2}} \cdot \frac{\sum_{w \in W_K} (-1)^w \cdot e^{w(\mu + \rho_K)}(x)}{\prod_{\alpha \in \Delta^+(\kg, \kt)}\left( e^{\frac{\alpha}{2}} - e^{\frac{-\alpha}{2}} \right)(x)}.
\]	
In particular, the index paring 
\[
\tau_x \left([P^E_t]\right)  \in \mathbb{C} .
\]
is independent of $t$. 
\end{theorem}

\begin{proof} Here we provide two different proofs. \\

\noindent{\bf Proof I:} By Theorem \ref{t0 thm}, we know that 
\[
\tau_x \left([P^E_t]\right)  = (-1)^{\frac{\dim \kp}{2}} \cdot \frac{\sum_{w \in W_K} (-1)^w \cdot e^{w(\mu + \rho_K)}(x)}{\prod_{\alpha \in \Delta^+(\kg, \kt)}\left( e^{\frac{\alpha}{2}} - e^{\frac{-\alpha}{2}} \right)(x)},
\]
for all $t >0$. Applying Theorem \ref{contin thm}, 
\[
\lim_{t\to 0}\tau_x \left([P^E_t]\right) = \tau_x \left([P^E_0]\right).
\]

\noindent{\bf Proof II.} We compute $\tau_x \left([P^E_0]\right)$ directly. To be more precise, let $f^\pm \in C_0(K \ltimes \kp)$ be the smoothing kernel of $e^{-D^{E}_{t,\mp}D^{E}_{t,\pm}}$. One can compute that the Fourier transforms
\[
\widehat{f}_\pm(z) = e^{-|z|^2} \cdot \id_{S^\pm_\kp\otimes E} \in  C_0\left(\widehat{\kp}, \bK \left(L^2(K)\right) \right)^K, \quad z \in \widehat{\kp}
\]
respectively. By Proposition \ref{prop tau},  we have that 
\[
\begin{aligned}
\tau_x(f^+-f^-)  = &	\frac{\chi_{S_\kp^+\otimes E}(x) - \chi_{S_\kp^-\otimes E}(x) }{\left(\chi_{S_\kp^+}(x) -\chi_{S_\kp^-}(x) \right)^2}\\
=&\frac{\chi_E(x)}{\left(\chi_{S_\kp^+}(x) -\chi_{S_\kp^-}(x) \right)}.
\end{aligned}
\]
The theorem follows immediately from Weyl character formula for $K$. 

\end{proof}

It is crucial that the group element $x$ in Theorem \ref{t0 thm} is a regular element. In the following we show that the pairing between the $L^2$-trace $\tau_e$ and $K_0(C^*_r(G))$ does change with respect to $t$. 

\begin{theorem}\label{thm:L2}
For the $L^2$-trace, we have that 
\[
\tau_e \left([P^E_t]\right) = t^{\dim \kp}\cdot \tau_e \left([P^E]\right). 
\]	
\end{theorem}
\begin{proof}
We follow the approach in \cite{CM}. First fix a volume form 
\[
\omega \in \Lambda^{\dim \kp} \kp^*.
\]
For any top degree form $\lambda \in \Lambda^{\dim \kp} \kp^* $, we define $\lambda([\kp])\in \mathbb{C}$ by
\[
\lambda = \lambda([\kp]) \cdot \omega. 
\] 
Then we identify the highest weight $\mu$ of $E$ as an element in $\kg^*$ and define 
\[
\beta_\mu \in \Lambda^2 \kp^*
\] 
by the following
\[
\beta_\mu(X, Y) = \mu\left([X, Y]_\kg\right), \quad X, Y \in \kp. 
\] 
By \cite[Proposition 7.1 A]{CM}, 
\begin{equation}
\label{CM}
\begin{aligned}
&\tau_e \left([P^E]\right) \\
=& \text{formal degree of the discrete series representation determined by $E$}\\
=&\frac{1}{n\cdot 2^n} \left(\Lambda^n \beta_\mu\right)\left([\kp]\right). 
\end{aligned}
\end{equation}
where $n = \frac{\dim \kp}{2}$. 

The family of Lie algebras $\kg_t$ can be identified with $\kg$ by the following map: for any $X = X_\kk + X_\kp \in \kk \oplus \kp$, we define 
\[
X^t = X_\kk + \frac{1}{t}\cdot X_\kp \in \kg_t. 
\]
One can check that 
\begin{equation}
\label{bracket t}
\begin{cases}
[X, Y]_{\kg_t} = 	[X, Y]_\kg & X, Y \in \kk,\\
[X, Y]_{\kg_t} = 	t^2 \cdot [X, Y]_\kg & X, Y \in \kp,\\
[X, Y]_{\kg_t} = 	[X, Y]_\kg & X \in \kk, Y \in \kp. \\
\end{cases}
\end{equation}
Applying the formula (\ref{CM}) to $\kg_t$, we obtain the formula of $\beta^t_\mu$ from (\ref{bracket t}) 
\[
\beta_\mu^t = t^2 \cdot \beta_\mu.  
\]
Therefore, we conclude
\[
\tau_e \left([P^E_t]\right) = t^{\dim \kp}\cdot \tau_e \left([P^E]\right). 
\]
\end{proof}

Theorem \ref{thm:L2} suggests that in general for a singular element $x$, $\tau_x([P^E_t])$ changes with respect to $t$ and the vanishing order  depends on $x$.

\bibliographystyle{amsalpha}

\end{document}